\documentclass[a4paper,10pt]{article}
\usepackage[final]{changes}
\usepackage[utf8]{inputenc}
\usepackage{amsmath}
\usepackage{amssymb}
\usepackage{fontenc}
\usepackage{graphicx}
\usepackage{amsthm}
\usepackage{enumitem}
\usepackage{algorithm}
\usepackage{algorithmic}
\usepackage{xcolor}

\newtheorem{lemma}{Lemma}
\newtheorem{theorem}{Theorem}

\newtheorem{definition}{Definition}
\newtheorem{remark}{Remark}
\providecommand{\keywords}[1]
{
  \small	
  \textbf{\textit{Keywords---}} #1
}

\title{Canonical form of modular hyperbolas with an application to integer factorization}
\author{Juan Di Mauro$^{1}$,$^{2}$\\
        \small $^{1}$CONICET \\
        \small $^{2}$ICC-University of Buenos Aires \\
        \small jdimauro@dc.uba.ar
}
\date{} 

\begin{document}

\maketitle

\begin{abstract}

For a composite $n$ and an odd $c$ with $c$ not dividing $n$, the number of solutions 
to the equation $n+a\equiv b\mod c$ with $a,b$ quadratic residues modulus $c$ 
is calculated. We establish a direct relation with those modular solutions and the distances between points of a modular hyperbola. Furthermore, for certain composite moduli $c$, an asymptotic formula for quotients between the number of solutions and $c$ is provided. Finally, an algorithm for integer factorization using such solutions is presented.
\end{abstract}
\section{Introduction}
For $n\in\mathbb{Z}$ and $p$ an odd prime, the set of integers 
\[\mathcal{H}_{n,p}=\{(x,y)\mid xy\equiv n \mod p,\ 0\leq x,y<p\}\]
defines a modular hyperbola of $n$ modulus $p$. If $p$ and $n$ are coprimes, the set $\mathcal{H}_{n,p}$ has $\phi(p)=p-1$ points. 

The above definition can be naturally extended to composite $c$ and the set $\mathcal{H}_{n,c}$. Regarding $\mathcal{H}_{n,c}$ as a subset of the Euclidean plane, the set of distances between points in the hyperbola is
\[\mathcal{D}_{n,c}=\{|x-y|\mid x,y\in \mathcal{H}_{n,c}\}\]

It is therefore not surprising to see that the set of distances and the set of solutions to the equation 
\begin{equation}
\begin{array}{cc}
\label{eq:targets}
n+a\equiv b\mod c & a,b\in\mathbb{Z}_c^2\cup\{0\} 
\end{array}
\end{equation}
are related. The latter form, however, can be seen as more convenient for some applications, as shown in section \ref{sec:fac}. For the case of composite odd numbers, the number of solutions of (\ref{eq:targets}) can be precisely calculated. As explained in section \ref{sec:tauO}, the number of solutions to (\ref{eq:targets}) is small with respect to $c$ for certain composite $c$, and that very fact can be directly used to build an integer factoring algorithm.

In usual notation, $(\frac{n}{p})$ stands for the Legendre symbol of $n$ modulus a prime $p$. In \cite{sh:1} Theorem 1, the following result is provided, that will become handy later on.
\begin{theorem}
 Let $p$ be a odd prime and $gcd(n,p)=1$, then
 \[\mid \mathcal{D}_{n,p}\mid=\Bigg\{\begin{array}{cc}
  \frac{p-1}{4}+\added{(1+(\frac{n}{p}))/2}&\textnormal{ if }p\equiv 1\mod 4\\
  \frac{p-3}{4}+1&\textnormal{ if }p\equiv 3\mod 4
 \end{array}\]
 \label{teo:TD}
\end{theorem}
The geometric intuition behind this Theorem is simple: in the square \replaced{$[0,p-1]^2$}{$[1,p-1]^2$} the points of the modular hyperbola are symmetric with respect to the lines given by $y=x$ and $y+x=p$. It is not too difficult to see that the elements of $\mathcal{H}_{n,p}$ lying on the symmetry axis $x=y$ are the square roots of $n$ modulus $p$, and those lying on the symmetry axis $y+x=p$ are the square roots of $-n$ modulus $p$. It therefore follows that such points exist if and only if $(\frac{n}{p})=1$ in the first case, and $(\frac{-n}{p})=1$ in the second.

For a given $u$, the subset of points of the hyperbola with distance $u$ is named $A_u$, that is
\[A_u=\{(x,y)\in \mathcal{H}_{n,p}\mid u=|x-y|\}.\]

As previously observed, if $A_u$ is not empty then it contains at least $2$ points. The Lemma below characterizes the set $A_u$.

\begin{lemma}
 Let $A_u$ be as above.
 \begin{enumerate}
  \item If $y$ is a root of $Y(Y-u)\equiv n\ \mod p$ then $(y-u,y)\in A_u$
  \label{lema:cond1}
  \item If $y$ is a root of $Y(Y+u)\equiv n\ \mod p$ then $(y+u,y)\in A_u$
  \label{lema:cond2}
  \item If $(x,y)\in A_u$ then $y$ is a root of the equations in \ref{lema:cond1}. or \ref{lema:cond2}.
 \end{enumerate}
\end{lemma}
\begin{proof}
The calculations are elementary and therefore omitted. 
\end{proof}

\replaced{It is immediately clear that $|A_u|=4$ provided that $(\frac{4n+u^2}{p})=1$. In case $4n+u^2\equiv 0 \mod p$, we have $|A_u|=2$ and finally $A_u=\emptyset$ if $(\frac{4n+u^2}{p})=-1$.}{It is immediately clear that $|A_u|=4$ provided that $4n+u^2\not\equiv 0 \mod p$ and $u\neq 0$, and otherwise $|A_u|=2$.}

When dividing the square \replaced{$[0,p-1]^2$}{$[1,p-1]^2$} into four regions $R_1,\ldots,R_4$, according to the symmetry lines $y=x$ and $x+y=p$ not two points of $A_u$ belong to the same region, and that means there exists a correspondence between $u$ and the points in one region. For simplicity, lets consider only the region 
\[R_1=\{(x,y)\in\mathcal{H}_{n,p}\mid \ 0\leq x<p, 0\leq y < \min(x,p-x) \}\]
and
\[\bar{R}_1=\{ (x,y)\in\mathcal{H}_{n,p} \mid\ 0\leq x<p, 0\leq y \leq \min(x,p-x) \}\]
\begin{lemma}
\label{lemma:RD}
For $p$ an odd prime, let $n$ be an integer with $gcd(n,p)=1$, \deleted{$(\frac{n}{p})=-1$ and $(\frac{-n}{p})=-1$}. Then, there is a bijective correspondence between the points of \replaced{$\bar{R}_1$}{$R_1$} and $\mathcal{D}_{n,p}$.
\end{lemma}
\begin{proof}

\deleted{For a given pair $(x,y)\in\mathcal{H}_{n,p}$, let $u=|x-y|\in\mathcal{D}_{n,p}$ and define $A$ as the set that contains $(x,y)$ and all points that are symmetric to $(x,y)$}

\deleted{Clearly, $A\subset A_u$ hence all the elements in $A$ are different, as implied by $gcd(n,p)=1$, $(\frac{n}{p})=-1$ and $(\frac{-n}{p})=-1$, so $A=A_u$.}

\added{Let $\bar{\psi}:\bar{R}_1\rightarrow \mathcal{D}_{n,p}$ be the application defined by $\bar{\psi}(x,y)=|x-y|$.}

\added{If $u\in \mathcal{D}_{n,p}$ then there is a pair $(x,y)\in\mathcal{H}_{n,p}$ with $|x-y|=u$. If $(x,y)\in \bar{R}_1$ then $\bar{\psi}(x,y)=u$. On the other hand, if $(x,y)\not\in \bar{R}_1$ then $f_1$, $f_2$ or a composition of both carries $(x,y)$ to a point $(x',y')\in \bar{R}_1$ with $|x-y|=|x'-y'|=u$. That proves the surjectivity}

\added{Note that $\bar{\psi}$ restricted to $R_1$ is injective because, if $(x_1,y_1), (x_2,y_2)\in R_1$ with $\bar{\psi}(x_1,y_1)=|x_1-y_1|=\bar{\psi}(x_2,y_2)=|x_2-y_2|=u$ then 
\[A_u=\{(x_1,y_1),(y_1,x_1),(p-x_1,p-y_1),(p-y_1,p-x_1)\}\]
and $(x_1,y_1), (x_2,y_2) \in A_u $. But $A_u$ has only one element of $R_1$, so it must be $(x_1,y_1)=(x_2,y_2)$.}

\added{It only remains to consider the injectivity in the points of $\bar{R_1}$ that are in the borders.}

\added{Lets suppose $\bar{\psi}(x_1,y_1)=\bar{\psi}(x_2,y_2)$. }

\added{If for $(x_1,y_1)\in\bar{R}_1$ we have $x_1=y_1$ then $u=|x_1-y_1|=|x_2-y_2|=0$. Exactly one element of the set $A_0=\{(x_1,x_1),(p-x_1,p-x_1)\}$ is in $\bar{R_1}$, so  $(x_1,y_1)=(x_2,y_2)=(x_1,x_1)$}

\added{Otherwise if $(x_1,y_1)\in\bar{R}_1$ satisfies $y_1=p-x_1$, then in that case for $u=|x_1-(p-x_1)|=|2x_1-p|$ it holds $A_u=\{(x_1,p-x_1),(p-x_1,x_1)\}$ and exactly one point of $A_u$ is in $\bar{R_1}$, so $(x_1,y_1)=(x_1,p-x_1)=(x_2,y_2)$.}

\end{proof}
\deleted{If $u=0$ and thus $(\frac{n}{p})=1$ then $A_u$ has only two points, $(x,x)$ and $(p-x,p-x)$, with $x$ the square root of $n$ modulus $p$ such that $x<p/2$. In that case, $(x,x)$ is the only element of $A_u$ in $\bar{R}_1$.}

\deleted{If $(\frac{-n}{p})=1$ and $y$ is the square root of $-n$ modulus $p$ with $y>p/2$, then for $u=p-2y$ and we have that $A_u=\{(y,p-y),(p-y,y)\}$, so $(y,p-y)$ is the only point of $A_u$ in $\bar{R}_1$.}

\deleted{The previous observations and the preceding Lemma shows that if $gcd(n,p)=1$ there is a one-to-one correspondence between the points of $\mathcal{D}_{n,p}$ and $\bar{R}_1$.}
 
\section{The canonical form of a modular hyperbola}
The motivation for the following definition can be found in the Fermat method to factor a composite $n$. In such case the factors of $n$ can be obtained by looking at the solutions of the Diophantine equation
\begin{equation}
n+x^2=y^2 .
\label{eq:tar:fermat}
\end{equation}

If (\ref{eq:tar:fermat}) is reduced modulus $c$ for some $c$ of modest size, the modular solutions of the equation can be obtained without much effort. It is enough (and naive) to try with half of the elements $\bar{x}\in \mathbb{Z}_c$ and keep the pairs $(\bar{x}^2,n+\bar{x}^2\mod c)$ for every $\bar{x}$ that satisfy $n+\bar{x}^2\mod c\in(\mathbb{Z}_c)^2\deleted{\cup \{0\}}$. Furthermore, if $x^*,y^*$ satisfies (\ref{eq:tar:fermat}) then clearly $x^*$ and $y^*$ must be equivalent modulus $c$ to one pair of solutions modulo $c$. This very fact motivates the following definition, which is due to Scolnik \cite{hugo}.

\begin{definition}
A \emph{target} for $n$ is a triplet $(a,b,c)$ of non-negative integers with $a,b\in(\mathbb{Z}_c)^2\deleted{\cup \{0\}}$ holding $n+a\equiv b\mod c$
 \label{def:target}
\end{definition}

It will also be the case that $a,b$ is a target for $n$ modulus $c$, meaning that $(a,b,c)$ is a target for $n$. In addition, if $(a,b,c)$ is a target for $n$ and $a\equiv {x^*}^2\mod c$, then for some $0\leq\alpha\leq c-1$ with $\alpha^2\equiv a\mod c$ there is an integer $z$ such that
 \begin{equation}
 {x^*}^2=(\alpha+cz)^2
 \label{eq:parab:x}
\end{equation}

Let $T(n,c)$ denote the set of targets for $n$ modulus $c$, that is
\[T(n,c)=\{(a,b,c)\mid a,b\in(\mathbb{Z}_c)^2\deleted{\cup \{0\}}\textnormal{ and } n+a\equiv b\mod c\} \]
and let $\tau(n,c)$ be the number of elements in $T(n,c)$. For an odd prime $p$ that does not divide $n$, it is expected that $\tau(n,p)$ and $\mid \mathcal{D}_{n,p}\mid$ are equal.

Although the proof can be stated by elementary means, it is more convenient to use a special case of a Jacobi sum. To do so, we write $N(x^2= a)$ as the number of solutions of $x^2\equiv a\mod p$ with $0\leq a\leq p-1$. Unless stated otherwise, from now on all elements should be in $\mathbb{Z}_p$ hence the solutions to $x^2\equiv a\mod p$ should be understood as the elements of $\mathbb{Z}_p$ that satisfy the former congruence.

The Lemma below is rather intuitive if one considers the canonical form of an hyperbola in the Euclidean space and the definition of $\mathcal{H}_{n,p}$.
\begin{lemma}
\label{lemma:nsol}
For an integer $n$ and an odd prime $p$ with $p\nmid n$, the number of solutions to $n+x^2\equiv y^2\mod p$ is $p-1$.
\end{lemma}
\begin{proof}
 The number of solutions to $n+x^2\equiv y^2\mod p$ can be written as
 \[N(n+x^2= y^2)=\sum_{n+a\equiv b\mod p}N(x^2= a)N(y^2= b)\]
 It is easy to see that $N(x^2= a)=1+\big(\frac{a}{p}\big)$, so the previous equation becomes
 \[\sum_{n+a\equiv b\mod p}\Big(1+\Big(\frac{a}{p}\Big)\Big)\Big(1+\Big(\frac{b}{p}\Big)\Big).\]
or
\[\sum_{a\in\mathbb{Z}_p}1+\sum_{a\in\mathbb{Z}_p}\Big(\frac{a}{p}\Big)+\sum_{b\in\mathbb{Z}_p}\Big(\frac{b}{p}\Big)+\sum_{n+a\equiv b\mod p}   \Big(\frac{a}{p}\Big)\Big(\frac{b}{p}\Big).\]

It is well known that for a non-trivial character $\chi$ over a field $\mathbb{F}_p$ the sum $\sum_{a\in\mathbb{F}_p^*}\chi(a)$ is $0$. As the Legendre symbol is a character over $\mathbb{Z}_p^*$, the second and third summations are $0$. 

On the other hand, if we rename $a=-na'$ and $b=nb'$ we get
\begin{align*}
 \sum_{n+a\equiv b}\Big(\frac{a}{p}\Big)\Big(\frac{b}{p}\Big)&=\sum_{a'+b'\equiv 1}\Big(\frac{-na'}{p}\Big)\Big(\frac{nb'}{p}\Big)\\
 &=\Big(\frac{-n}{p}\Big)\Big(\frac{n}{p}\Big)\sum_{a'+b'\equiv 1}\Big(\frac{a'}{p}\Big)\Big(\frac{b'}{p}\Big)\\
 &=(-1)^{\frac{p-1}{2}}\sum_{a'+b'\equiv 1}\Big(\frac{a'}{p}\Big)\Big(\frac{b'}{p}\Big)
\end{align*}
(the congruence is modulus $p$). This is a Jacobi sum $J(\chi,\chi)$ with the Legendre symbol as the character $\chi$. Furthermore, as the Legendre symbol is a character of order $2$, meaning that $J(\chi,\chi)=J(\chi,\chi^{-1})$, the special case 
\[J(\chi,\chi^{-1})=-(-1)^{\frac{p-1}{2}}\] holds (see Theorem 1 in chapter 8, section 3 of \cite{ireland}). Finally,
\[N(n+x^2= y^2)=p+(-1)^{\frac{p-1}{2}}(-(-1)^{\frac{p-1}{2}})=p-1\]

\end{proof}

\begin{theorem}
Let $p$ be an odd prime with $gcd(p,n)=1$, then
 \[\tau(n,p)=\Bigg\{\begin{array}{cc}
  \frac{p-1}{4}+\replaced{(1+(\frac{n}{p}))/2}{(\frac{n}{p})}&\textnormal{ if }p\equiv 1\mod 4\\
  \frac{p-3}{4}+1&\textnormal{ if }p\equiv 3\mod 4
 \end{array}\]
 hence $\tau(n,p)=\mid \mathcal{D}_{n,p}\mid$.
 \label{teo:tau}
\end{theorem}
\begin{proof}
The proof goes by stages. If $p\equiv 1\mod 4$ and $\Big(\frac{n}{p}\Big)=-1$ then there is no target $(a,b,p)$ with $a=0$ or $b=0$, so every target $(a,b,p)$ gives four solutions to the equation $n+x^2\equiv y^2\mod p$. That is,
\[4\tau(n,p)=p-1.\]
by Lemma \ref{lemma:nsol}. \added{Hence, 
\[\tau(n,p)=\frac{p-1}{4}.\]}

If $p\equiv 1\mod 4$ and $\Big(\frac{n}{p}\Big)=1$ then also $\Big(\frac{-n}{p}\Big)=1$, there is one target of the form $(0,b,p)$ and one target of the form $(a,0,p)$ with $a\neq 0$ and $b\neq 0$, each one giving two solutions to $n+x^2\equiv y^2\mod p$ and the rest providing four solutions. It therefore means that
\[4(\tau(n,p)-2)+4=p-1.\]
\added{So
\[\tau(n,p)=\frac{p-1}{4}+1.\]}
Finally, if $p\equiv 3\mod 4$ exactly one of the pairs $\{n,-n\}$ is a quadratic residue modulus $p$, meaning there is only one target $(a,b,p)$ of $n$ with $a=0$ or $b=0$. That target gives two solutions to $n+x^2\equiv y^2\mod p$, and each one of the rest gives four. So,
\[4(\tau(n,p)-1)+2=p-1\]
\added{and from that
\[\tau(n,p)=\frac{p-3}{4}+1\]
}
\end{proof}
\begin{remark}
\label{rem:taumulti}
If $s$ and $t$ are coprime, then $\tau(n,st)=\tau(n,t)\tau(n,s)$ following from the Chinese Remainder Theorem, so $\tau(n,c)$ is multiplicative as a function of $c$.
\end{remark}
Furthermore, $\tau$ can be calculated for moduli that are powers of odd primes.
\begin{theorem}
\label{teo:tau:pk}
Let $p$ be an odd prime, $n$ an integer with $gcd(p,n)=1$ and $k>0$.
 \begin{enumerate}
  \item If $(a,b,p^k)$ is a target for $n$ with $a\neq 0$ and $b\neq 0$, then there are $p$ targets $(a',b',p^{k+1})$ of $n$ with $a\equiv a'\mod p^k$ and $b\equiv b'\mod p^k$.
  \item If $(0,b,p^k)$ is a target for $n$, the number of targets $(a',b',p^{k+1})$ for $n$ with $0\equiv a'\mod p^k$ and $b\equiv b'\mod p^k$ is
  \begin{enumerate}
   \item  $|(\mathbb{Z}_p)^2|\deleted{+1}$ for $k$ even.
   \item $1$ for $k$ odd.
  \end{enumerate}
   \item If $(a,0,p^k)$ is a target for $n$, the number of targets $(a',b',p^{k+1})$ for $n$ with $a\equiv a'\mod p^k$ and $0\equiv b'\mod p^k$ is
  \begin{enumerate}
   \item  $|(\mathbb{Z}_p)^2|\deleted{+1}$ for $k$ even.
   \item $1$ for $k$ odd.
  \end{enumerate}
  \end{enumerate}
\end{theorem}
\begin{proof}
 \begin{enumerate}
  \item If $g\not\equiv 0\mod p$ is a quadratic residue modulus $p^k$, then 
  $g+tp^k$ is a quadratic residue modulus $p^{k+1}$ for $0\leq t \leq p$. The set 
  \[\{(a+tp^k,b+t'p^k,p^{k+1})\mid t'=(n+a-b)/p^k+t,\ 0\leq t < 0\}\]
  is a set of targets for $n$.
  \item If $k$ is even, $tp^k$ is a quadratic residue modulus $p^{k+1}$ provided 
  $t\in(\mathbb{Z}_p^2)\deleted{\cup \{0\}}$. Besides, $b+t'p^k$ is quadratic residue of $p^{k+1}$ for $0\leq t'<p$. The set
  \[\{(tp^k,b+t'p^k,p^{k+1})\mid t'=(n-b)/p^k+t,\ 0\leq t < 0\}\]
  is a set of targets for $n$.
  
  If $k$ is odd, $tp^k$ is a quadratic residue of $p^{k+1}$ only for $t=0$.
  \item Analogous to the previous case.
 \end{enumerate}

\end{proof}
For brevity, let $s_p(n)=(1+(\frac{n}{p}))/2+(1+(\frac{-n}{p}))/2$. 
\begin{theorem}
\label{teo:tau:pk:bis}
 If $p$ is an odd prime and $p\nmid n$, then for $k>1$
 \[\tau(n,p^{k+1})=\Bigg\{\begin{array}{cc}
  [\tau(n,p^{k})-s_p(n)]p+s_p(\deleted{-}n)(|(\mathbb{Z}_p)^2|+1) &\textnormal{ if }k\textnormal{ is even }\\
  
  [\tau(n,p^{k})-s_p(\deleted{-}n)]p+s_p(n)&\textnormal{ if }k \textnormal{ is odd}
 \end{array}\]
 \label{teo:taupotencia}
\end{theorem}
\begin{proof}
 Follows immediately from Theorem \ref{teo:tau:pk}
\end{proof}
From the theorems \ref{teo:tau}, \ref{teo:tau:pk:bis} and Remark \ref{rem:taumulti} the number of targets for $n$ modulus an odd composite $c$ whose prime factors are known can be easily calculated.

\section{Relation between targets and $\mathcal{D}_{n,c}$}
\label{sec:tauO}
As the number of solutions to $n+x^2\equiv y^2\mod p$ and the number of points in $\mathcal{H}_{n,p}$ are the same, it can be easily seen that the targets for $n$ modulus $p$ and the set of distances between points of a modular hyperbola $\mathcal{H}_{n,p}$ are related.
\begin{theorem}
\label{teo:biyec}
 Let $p$ be an odd prime with $\gcd(n,p)=1$\deleted{, $(\frac{n}{p})=-1$ and $(\frac{-n}{p})=-1$}.\replaced{There}{ Then, there} is a correspondence between the elements of $T_{n,p}$ and $\mathcal{D}_{n,p}$.
\end{theorem}
\begin{proof}
\deleted{For simplicity, instead of \replaced{$\bar{R}_1$}{$R_1$} consider.} \deleted{As $\gcd(n,p)=1$, $(\frac{n}{p})=-1$ and $(\frac{-n}{p})=-1$, there are no targets of the form $(0,b,p)$ or $(a,0,p)$ and there are no points of $\mathcal{H}_{n,p}$ in the boundaries of $R_1$.}
 
There is indeed a correspondence between the elements of $T_{n,p}$ and the elements of \replaced{$\bar{R}_1$}{$R_4$} as it can be seen by the mapping \deleted{from $R_4$ to $T_{n,p}$} $\gamma_1:\bar{R}_1\rightarrow T_{n,p}$, given by
 \[(x,y)\mapsto (2^{-2}(x-y)^2,2^{-2}(x+y)^2,c)=(a,b,c)\]
 and the mapping \deleted{from $T_{n,p}$ to $R_4$} $\gamma_2:T_{n,p}\rightarrow \bar{R}_1$ , given by
 \[(a,b,c)\mapsto (x,y)\] such that \replaced{$A_u \cap \bar{R}_1=\{(x,y)\}$}{$A_u \cap R_1=\{(x,y)\}$} with $u=2\alpha$, $\alpha^2\equiv a\mod p$ and $\alpha<p/2$ (by the proof of Lemma \ref{lemma:RD}, this map is well-defined). From Lemma \ref{lemma:RD}, the result follows. 
\end{proof}

Besides their algorithmic implications, theorems \ref{teo:tausobrec} and \ref{teo:tausobrec2} provide bounds that show the asymptotic behavior of $\tau$ with respect to $c$ for certain composite moduli. 

In what follows, $p$ stands for an odd prime. The symbols $O$ and $\Theta$ are used, as customary, to describe asymptotic bounds. A simple result (following from one of Mertens' Theorems; see Corollary 2.10 in \cite{hit} and Theorem 5.13 in \cite{shoup}) and that will become useful later is 
$\prod_{p\leq B} \frac{p+1}{p}=\Theta(\log{B})$. As usual, $\pi$ is used for the prime-counting function, so $\pi(B)$ is the number of primes up to $B$.
\begin{theorem}
If $c=\prod_{p\leq B} p$ and for all $p\leq B$ with $p\equiv1 \mod 4$ holds $\big(\frac{n}{p}\big)=-1$ then  \[\frac{\tau(n,c)}{c}=O(4^{-\pi(B)}\log{B})\]
 \label{teo:tausobrec}
\end{theorem}
\begin{proof}
 By Remark \ref{rem:taumulti}, 
 \[\tau(n,c)=\prod_{p\leq B} \tau(n,p)\]
and from the definition of $\tau$ and the assumptions of the above Theorem, it follows that $\tau(n,p)\leq (p+1)/4$ for all $p\leq B$. Therefore, 
\[\tau(n,c)=\prod_{p\leq B} \tau(n,p)\leq 4^{-\pi(B)}\prod_{p\leq B} p+1.\]
Dividing by $c$ and applying $\prod_{p\leq B} \frac{p+1}{p}=\Theta(\log{B})$ the desired result immediately follows.
\end{proof}

\begin{theorem}
 If $c=\prod_{p\leq B} p$ then \[\frac{1}{c}\prod_{2<p\leq B}\Big(\tau(n,p)-\replaced{(1+\big(\frac{n}{p}\big))/2}{\big(\frac{n}{p}\big)}\Big)=O(4^{-\pi(B)}\log{B})\]
 \label{teo:tausobrec2}
\end{theorem}
\begin{proof}
  Note that $\tau(n,p)-\replaced{(1+\big(\frac{n}{p}\big))/2}{\big(\frac{n}{p}\big)}\leq (p+1)/4$. The proof is identical as the one given above.
\end{proof}

\section{An application to factoring}
\label{sec:fac}
We now want to briefly sketch an algorithm for integer factorization using targets. The complexity analysis of the algorithm is extense and will be presented in a forthcoming paper; hence we exhibit here the basic idea. Although the algorithm is not of practical use nowadays, given that its running time is exponential (roughly $O(n^{1/3})$), there are many subexponential algorithms that were born from previous exponential ones, and besides that, it shows one possible way for using the elements in $\mathcal{D}_{n,c}$ to factor $n$. 

Let $n=pq$ be composite with $p,q$ primes with $|p-q|$ not too big, so that if $x^*,y^*$ are solutions to $n+x^2=y^2$ then $|x^*|<n^{1/2}$. 

As it was suggested by theorems \ref{teo:tausobrec} and \ref{teo:tausobrec2}, the idea is to let $c'\cdot c=p_1\ldots p_m$ be the product of the first consecutive odd primes (with $c'=p_1\ldots p_r$, $0<r<m$) and observe that, for some pair of targets $(a,b,c)$ and $(a',b',c')$ it holds that
\[{x^*}^2=a+tc=a'+uc'\]
for integers $t,u$. Clearly then, if $(a,b,c)$ is a right target for $n$ for some $\alpha$ square root of $a$ modulus $c$, then there is an integer $z$ bounded by $|z|<n^{1/2}/c$ such that
\[{x^*}^2=(\alpha+cz)^2.\]
But also note that $(\alpha+cz)^2\equiv a'\mod c'$, hence if $z_1,\ldots,z_l$ are the solutions of  $(\alpha+cz)^2\equiv a'\mod c'$ with $0\leq z_i\leq c'$, in which case 
\[{x^*}^2=(\alpha+cz_i+c\cdot c'\cdot k)^2\]
for some $i=1,\ldots,l$, with $k$ an integer bounded by $|k|\leq \lceil n^{1/2}/(c' \cdot c)\rceil$. The search is conducted over $k$, and as $m$ grows the bound gets narrowed. An appropriate choice is to take $r=\lfloor m/2 \rfloor$ and $c\cdot c'$ as the maximum product of consecutive odd primes that satisfy
\[n^{1/2}/(c' \cdot c)\geq 1\]

The complete procedure is detailed as Algorithm \ref{algo:factortargets}.
 \begin{algorithm}[H]
 
 \caption{Factoring integers using targets}\label{algo:factortargets} 
 \begin{algorithmic}[1]
\REQUIRE $n$, a composite to factor. $m$ a positive integer such that $p_1\ldots p_m$ is the maximum product of odd primes with $n^{1/2}/(p_1\ldots p_m)\geq 1$.
  \STATE Choose $0<r<m$, let $c'\leftarrow p_1\ldots p_r$, $c\leftarrow p_{r+1}\ldots p_m$.
  \STATE $z_{max}\leftarrow \lceil n^{1/2}/c \rceil$
  \STATE $k_{max}\leftarrow \lceil n^{1/2}/(c\cdot c') \rceil$
  \STATE $k_{min}\leftarrow -k_{max}$.
  \FOR{every target $(a_i,b_i,c)$} \label{algo2:step:1} 
  \STATE Let $\{\alpha_{i,1},\ldots,\alpha_{i,k_i}\}$ be the square roots of $a_i$ modulus $c$.\label{algo2:step:2} 
  \STATE $d\leftarrow (c^{-1} \mod c')$.
  \FOR{$j=1,\ldots,k_i$}\label{algo2:step:3}
  \STATE $t_{i,j}\leftarrow d(\alpha_{i,1}-\alpha_{i,j})\mod c'$\COMMENT{Precalculations} \label{algo2:step:4}
  \STATE $\delta_{i,j}\leftarrow \alpha_{i,j}+ct_{i,j}$\label{algo2:step:5}
  \ENDFOR \label{algo2:step:6}
  \STATE $\theta_i\leftarrow d(\alpha_{i,1}-\alpha_{1,1})\mod c'$\label{algo2:step:7}
  \ENDFOR\label{algo2:step:8}
  \STATE Let $P(z)=(\alpha_{1,1}-cz)^2$.\label{algo2:step:9}
  \STATE $i\leftarrow 0$  \label{algo2:step:10}
  \FOR{$\ell=1,\ldots,\tau(n,c')$}\label{algo2:step:11}
  \STATE Take $(a'_{\ell},b'_{\ell},c')$, the target number $\ell$ of $n$ modulus $c'$\label{algo2:step:12}
  \FOR{$z=0,\ldots,c'-1$}\label{algo2:step:13}
  \IF{$(P(z)-a_{\ell}')\mod c'=0$}\label{algo2:step:14}
  \STATE $z_i\leftarrow z$, $i\leftarrow i+1$ \COMMENT{Initial \replaced{points}{dots}}\label{algo2:step:15}
  \ENDIF\label{algo2:step:16}
  \ENDFOR\label{algo2:step:17}
  \STATE $u\leftarrow i-1$ \label{algo2:step:18}
  \ENDFOR \label{algo2:step:19}
  \FOR{$s=1,\ldots,\tau(n,c)$}\label{algo2:step:20}
  \FOR{$t=1,\ldots,k_i$}\label{algo2:step:21}
  \STATE Define $x_{i,j}=\delta_{s,t}+c(\theta_s+c'\cdot j+z_i)$\COMMENT{Search for a solution}\label{algo2:step:22}
  \STATE $f\leftarrow \prod_{i=0}^u\prod_{j=k_{min}}^{k_{max}} \lfloor (x_{i,j}+n)^{1/2}\rfloor-x_{i,j} \mod n$\label{algo2:step:23}
  \STATE $g\leftarrow |\gcd(f,n)|$\label{algo2:step:24}
  \IF{$1<g<n$}\label{algo2:step:25}
  \RETURN $(g,n/g)$\label{algo2:step:26}
  \ENDIF\label{algo2:step:27}
  \ENDFOR\label{algo2:step:28}
  \ENDFOR\label{algo2:step:29}
 \end{algorithmic}

 \end{algorithm}
The algorithm was implemented in Python. After conducting several tests with small numbers
 (up to $30$ digits) it became evident that the number of iterations is near $O(\log{p_m}n^{1/3})$ (all the operations can be bounded by the cost of multiplying integers of size at most $\log{n}$ and taking the integer square root in line \ref{algo2:step:23}).

\end{document}